\documentclass[12pt]{amsart}
\usepackage{amssymb}
\usepackage{amscd}
\usepackage{amsmath}
\usepackage{enumerate}
\usepackage[pagebackref,colorlinks=true]{hyperref} 
\usepackage{tikz}
\usepackage{verbatim}

\newtheorem{theo}{Theorem}[section]
\newtheorem{lemma}[theo]{Lemma}
\newtheorem{cor}[theo]{Corollary}
\newtheorem{prop}[theo]{Proposition}

\newtheorem{defi}[theo]{Definition}

\newtheorem{example}[theo]{Example}

\newtheorem{question}[theo]{Question}
\def\e{{\epsilon}}
\def\a{{\alpha}}
\def\b{{\beta}}
\newcommand{\K}{{\mathbb{K}}}
\newcommand{\B}{{\mathbb{B}}}
\newcommand{\C}{{\mathbb{C}}}
\newcommand{\R}{{\mathbb{R}}}

\newcommand{\BS}{\mathbb{S}}
\newcommand{\BD}{\mathbb{D}}
\newcommand{\D}{\mathbb{D}}

\newcommand{\F}{\mathcal{F}}

\newcommand{\HH}{\mathcal{H}}

\begin{document}
\title[]{On deformations of isolated singularity functions}
\author{Aur\'elio Menegon}
\author{Miriam da Silva Pereira}

\thanks{The first author was supported by the {\it Conselho Nacional de Desenvolvimento Cient\'ifico e Tecnol\'ogico} (CNPq), Brazil.}
\address{Aur\'elio Menegon: Departamento de Matem\'atica - Universidade Federal da Para\'iba - Brazil.}
\email{aurelio.menegon@academico.ufpb.br}
\address{Miriam da Silva Pereira: Departamento de Matem\'atica - Universidade Federal da Para\'iba - Brazil.}
\email{msp@academico.ufpb.br}

\begin{abstract}
We study multi-parameters deformations of isolated singularity function-germs on either a subanalytic set or a complex analytic spaces. We prove that if such a deformation has no coalescing of singular points, then it has constant topological type. This extends some classical results due to L\^e \& Ramanujam (1976) and Parusi\'nski (1999), as well as a recent result due to Jesus-Almeida and the first author. It also provides a sufficient condition for a one-parameter family of complex isolated singularity surfaces in $\C^3$ to have constant topological type. On the other hand, for complex isolated singularity families defined on an isolated determinantal singularity, we prove that $\mu$-constancy implies constant topological type.
\end{abstract}
\maketitle

\section*{Introduction}  

The study of deformations of maps and spaces plays a central role in Singularity Theory and in Complex Geometry. For holomorphic isolated singularity function-germs on a complex affine space $\C^n$, the topological behavior of small perturbations is related to the Milnor number, which is an analytic invariant with a very clear topological and algebraic meaning. In fact, when $n\neq3$ there is a classical result due to L\^e \& Ramanujam \cite{LR} in 1976, which states that a family $f_s: (\C^n,0) \to (\C,0)$ is topologically trivial if and only if the Milnor number $\mu(f_s)$ is constant, for $s \in \C$ sufficiently small. 

The case $n=3$ is one of the main open problems in Singularity Theory, known as {\it the $\mu$-constant conjecture}. Here, topological triviality means that for each $s$ small enough there exists a homeomorphism of germs $h_s: (\C^n,0) \to (\C^n,0)$ that takes the germ of $V(f_s) := f_s^{-1}(0)$ at the origin onto the germ of $V(f_0):= f_0^{-1}(0)$ at the origin. 

In 1999, Parusi\'nski \cite{Pa} extended L\^e \& Ramanujam's result for $\K$-analytic isolated singularity families $f_s: (\K^n,0) \to (\K,0)$, with $\K=\C$ or $\R$, even when $n=3$, provided that the family depends linearly on the parameter $s \in \K$. This means that it has the form $f_s = f_0 + s \varphi$, for some $\K$-analytic isolated singularity function-germ $\varphi: (\K^n,0) \to (\K,0)$.

Recently, Jesus-Almeida and the first author \cite{MJ} studied $\K$-analytic families $f_s: (X,0) \to (\K,0)$ depending linearly on one parameter $s \in \K$, where $(X,0) \subset (\K^N,0)$ is the germ of a subanalytic set, when $\K=\R$, or a complex analytic space, when $\K=\C$. In particular, it was proved that if such a family has no coalescing of singular points, then it has constant topological type. Here, having no coalescing of singular points means that there exist real numbers $\delta>0$ and $\e>0$ such that $\Sigma(f_s) \cap \B_{\e} = \{0\}$ for every $s \in \D_\delta$, where $\Sigma(f_s)$ denotes the singular locus of $f_s$ (in the stratified sense, with respect to a given $(w)$-regular stratification) and $\B_\e$ denotes the closed ball around $0$ in $\K^N$. Also, having constant topological type means that for each $s \in (\K,0)$ there exists a homeomorphism of germs $h_s: \big( V(f_s),0 \big) \to \big( V(f_0),0 \big)$. Clearly, this is weaker than having topological triviality in the sense of L\^e \& Ramanujam.

The first goal of this paper is to study $\K$-analytic isolated singularity families $f_{s_1, \dots, s_k}: (X,0) \to (\K,0)$ depending linearly on $k$-many parameters $s_1, \dots, s_k \in \K$, with $k>1$. Then this will allow us to study one-parameter $\K$-analytic isolated singularity families $\mathtt f_s: (X,0) \to (\K,0)$ that depend polynomially on $s \in \K$, which means it has the form
$$\mathtt f_s = f_0 + s^{a_1} \varphi_1 + s^{a_2} \varphi_2 + \dots + s^{a_k} \varphi_k \, ,$$
with $0<a_1< \dots < a_k$. This will be done by setting 
$$f_{s_1, \dots, s_k} := f_0 + s_1 \varphi_1 + s_2 \varphi_2 + \dots + s_k \varphi_k \, ,$$
so that $\mathtt f_s = f_{s^{a_1}, \dots, s^{a_k}}$.

Our main results are the following:

In Section 1, we give an easy criterium to assure that a $\K$-analytic family $f_{s_1, \dots, s_k}: (X,0) \to (\K,0)$ has isolated singularity or even no coalescing of critical points. This is Proposition \ref{theo_is}. In Section 2, we prove:

\begin{theo} \label{theo_ctt}
Let $(X,0) \subset (\K^N,0)$ be the germ of either a subanalytic set, when $\K=\R$, or a complex analytic space, when $\K=\R$. Let $\mathcal S$ be a $(w)$-regular stratification of a representative $X$ of $(X,0)$. Let $(f_{s_1, \dots, s_k})$ be a family of $\K$-analytic isolated singularity function-germs on $(X,0)$ depending linearly on the parameters $s_1, \dots, s_k \in \K$. If $(f_{s_1, \dots, s_k})$ has no coalescing of singular points (with respect to $\mathcal S$) then it has constant topological type.
\end{theo}

As an important consequence, we have:

\begin{theo} \label{cor_ctt}
Let $\mathtt f_s: (X,0) \to (\K,0)$ be a one-parameter $\K$-analytic isolated singularity family depending polynomially on $s \in \K$. Let $f_{s_1, \dots, s_k}: (X,0) \to (\K,0)$ be the multi-parameter family associated to $(\mathtt f_s)$, defined as above. If $(f_{s_1, \dots, s_k})$ has no coalescing of singular points, then $(\mathtt f_s)$ has constant topological type.
\end{theo}

Corollary \ref{cor_ctt}  relates to the $\mu$-constant conjecture, as it provides a sufficient condition for a one-parameter family of complex isolated singularity surfaces in $\C^3$ to have constant topological type. Precisely, we have:

\begin{cor} \label{cor_mcc}
Let $\mathtt f_s: (\C^3,0) \to (\C,0)$ be an analytic isolated singularity family depending polynomially on $s \in \C$. Let $f_{s_1, \dots, s_k}: (\C^3,0) \to (\C,0)$ be the multi-parameter family associated to $(\mathtt f_s)$, as above. If $(f_{s_1, \dots, s_k})$ is $\mu$-constant, then $(\mathtt f_s)$ has constant topological type.
\end{cor}

The case when the family $\mathtt f_s: (\C^3,0) \to (\C,0)$ depends analytically on $s \in \C$ is analogous, by a well-known result of Samuel \cite{Sa} (see also \cite{LR}, Theorem 1.7).

In Section 3, we restrict our attention to the particular case of complex analytic isolated singularity families $f_{s_1, \dots, s_k}: (X,0) \to (\C,0)$ depending linearly on the parameter $s_1, \dots, s_k \in \C$, with $k \geq 1$, where $(X,0) \subset (\C^N,0)$ is the germ of a complex isolated determinantal singularity (IDS). In this case, there is a well-defined Milnor number $\mu(f_s)$ in the sense of \cite{BOT}.

We prove:

\begin{theo} \label{theo_3}
Let $(X,0) \subset (\C^N,0)$ be the germ of a complex isolated determinantal singularity. Let $f_{s_1, \dots, s_k}: (X,0) \to (\C,0)$ be a family of isolated singularity function-germs depending linearly on $s_1, \dots, s_k \in \C$, with $k \geq 1$. If $(f_{s_1, \dots, s_k})$ is $\mu$-constant, then it has constant topological type.
\end{theo}

As a consequence, like before, we have:

\begin{cor} \label{cor_mcc2}
Let $(X,0) \subset (\C^N,0)$ be the germ of a complex isolated determinantal singularity. Let $\mathtt f_s: (X,0) \to (\C,0)$ be an analytic isolated singularity family depending polynomially on $s \in \C$, and let $f_{s_1, \dots, s_k}: (X,0) \to (\C,0)$ be the multi-parameter family associated to $(\mathtt f_s)$, as above. If $(f_{s_1, \dots, s_k})$ is $\mu$-constant, then $(\mathtt f_s)$ has constant topological type.
\end{cor}

\section{Isolated singularity families and coalescing of singular points}  

Let $(X,0) \subset (\K^N,0)$ be as above, and let $\mathcal S = (\mathcal S_\a)_{\a \in \Lambda}$ be a $(w)$-regular stratification of a representative $X$ of $(X,0)$, in the sense of \cite{Ve}. Without losing generality, we can suppose that $\{0\}$ is a stratum of $\mathcal S$.

Let $\tilde f_0: (\K^N,0) \to (\K,0)$ be an analytic function-germ and let $f_0: (X,0) \to (\K,0)$ be the restriction of $\tilde f_0$ to $(X,0)$. Clearly, the singular locus $\Sigma(f_0)$ of $f_0$ is formed by $X \cap \Sigma(\tilde f_0)$, the intersection of $X$ with the critical set of $\tilde f_0$, together with the points $x \in X \setminus \Sigma(\tilde f_0)$ such that the smooth manifold $(\tilde f_0)^{-1}(\tilde f_0(x))$ intersects $\mathcal S_{\a(x)}$ not transversally at $x$ in $\K^N$, where $\mathcal S_{\a(x)}$ denotes the stratum of the stratification $\mathcal S$ that contains the point $x$. 

Notice that since we chose a stratification $\mathcal S$ such that $\{0\}$ is a stratum, the origin $0 \in X \subset \K^N$ is a singular point of $f_0$, even if $\tilde f_0$ is a submersion at $0 \in \K^N$.

We say that $f_0: (X,0) \to (\K,0)$ has an {\it isolated singularity} if there exists a positive real number $\e>0$ such that $\Sigma(f_0) \cap \B_\e = \{0\}$. This means that the restriction of $f_0$ to each stratum $\mathcal S_\a \neq \{0\}$ is a submersion.

Now, if $\tilde f_s: (\K^N,0) \to (\K,0)$ is a deformation of $\tilde f_0$ depending continuously on the parameter $s \in (\K^k,0)$, it induces a deformation $f_s: (X,0) \to (\K,0)$ of $f_0$ that obviously depends continuously on $s$. In the case when $f_0$ has an isolated singularity, the next lemma relates the critical set of $f_s$ with the critical set of $\tilde f_s$ inside some small neighborhood of $0 \in \K^N$, whenever $\|s\|$ is small enough.

We have:

\begin{lemma} \label{prop}
Suppose that $f_0: (X,0) \to (\K,0)$ has an isolated singularity, and let $\e>0$ be such that $\Sigma(f_0) \cap \overline\B_\e = \{0\}$. Then there exists $\delta>0$ sufficiently small such that for every $s \in \D_\delta$ one has that 
$$\Sigma(f_s) \cap \overline\B_\e = \left( \Sigma(\tilde f_s) \cap X \cap \overline \B_\e \right) \cup \{0\} \, .$$
\end{lemma}

\begin{proof}
Clearly, one has that $\Sigma(\tilde f_s) \cap X \subset \Sigma(f_s)$, so one inclusion is obvious. On the other hand, for every $x \in X \cap \overline\B_\e$ fixed, one has that either $x$ is the origin or $(\tilde f_0)^{-1}(\tilde f_0(x))$ is smooth at $x$ and it intersects $\mathcal S_{\a(x)}$ transversally at $x$ in $\K^N$. Since transversality is a stable property, there exists $\delta(x)>0$ sufficiently small such that for every $s \in \D_{\delta(x)}$ one has that either $x \in \Sigma(\tilde f_s) \cup \{0\}$ or $(\tilde f_s)^{-1}(\tilde f_s(x))$ is smooth at $x$ and it intersects $\mathcal S_{\a(x)}$ transversally at $x$ in $\K^N$. Moreover, there exists a small neighborhood $V_x$ of $x$ in $X \cap \overline\B_\e$ such that for every $x' \in V_x$ and for every $s \in \D_{\delta(x)}$ one has that either $x' \in \Sigma(\tilde f_s) \cup \{0\}$ or $(\tilde f_s)^{-1}(\tilde f_s(x'))$ is smooth at $x'$ and it intersects $\mathcal S_{\a(x)}$ transversally at $x'$ in $\K^N$. Therefore, since $X \cap \overline\B_\e$ is compact, there exists $\delta>0$ such that for every $x \in X \cap \overline\B_\e$ and for every $s \in \D_\delta$ one has that either $x \in \Sigma(\tilde f_s) \cup \{0\}$ or $x \notin \Sigma(f_s)$. Hence $\Sigma(f_s) \cap \overline\B_\e \subset \left( \Sigma(\tilde f_s) \cap X \cap \overline \B_\e \right) \cup \{0\}$.
\end{proof}

We say that a family $(f_s)$ as above is an {\it isolated singularity family} if for each $s \in \K$ with $\|s\|$ sufficiently small, the function-germ $f_s: (X,0) \to (\K,0)$ has isolated singularity. That is, if there exists $\delta>0$ such that for every $s \in \D_\delta$ there is $\e(s)>0$ such that $\Sigma(f_s) \cap \B_{\e(s)} = \{0\}$.

An isolated singularity family $(f_s)$ as above is said to have {\it no coalescing of singular points} if there exist real numbers $\delta>0$ and $\e>0$ such that $\Sigma(f_s) \cap \B_{\e} = \{0\}$ for every $s \in \D_\delta$.

As an immediate consequence of Lemma \ref{prop}, we have:

\begin{prop} \label{theo_is}
Suppose that $f_0: (X,0) \to (\K,0)$ has isolated singularity. Then:
\begin{itemize}
\item[$(i)$] If $(\tilde f_s)$ is an isolated singularity family, then $(f_s)$ is also an isolated singularity family.
\item[$(ii)$] If $(\tilde f_s)$ has no coalescing of critical points, then $(f_s)$ has no coalescing of singular points.
\end{itemize}
\end{prop}

\begin{example} \label{example_1}
Consider $X=\{x^2-y^2=0\} \subset \K^3$ with the stratification given by $\mathcal{S}_0= X \setminus \{x=y=0\}$ and $\mathcal{S}_1= \{x=y=0\}$. Consider the family of function-germs $f_s: (X,0) \to (\K,0)$ given by 
$$f_s(x,y,z)=x^3+y^4+z^2+sx^4 + s^2y^5 \, .$$
Since $f_0$ has an isolated singularity and since the family $(\tilde f_s)$ on $\K^3$ is an isolated singularity family with no coalescing of critical points, it follows from Theorem \ref{theo_is} that $(f_s)$ is an isolated singularity family with no coalescing of singular points.
\end{example}

\section{One-parameter families of isolated singularities}  

Let $(X,0) \subset (\K^N,0)$ and $\mathcal S = (\mathcal S_\a)_{\a \in \Lambda}$ be as above. Let $g_s: (X,0) \to (\K,0)$ be an isolated singularity family of $\K$-analytic function-germs depending linearly on the parameter $s \in \K$. This means that each $g_s$ has the form
$$g_s(z) = g_0(z) + s \varphi(z)  \, ,$$
where $g_0: (X,0) \to (\K,0)$ and $\varphi: (X,0) \to (\K,0)$ are $\K$-analytic isolated singularity function-germs.

Define the map-germ
$$
\begin{array}{cccc}
\F: & (X,0) & \to & (\K^2,0) \\
     & z & \mapsto & \big( g_0(z), \varphi(z) \big) \\
\end{array}
\, ,$$
and for each $s \in \K$ define the line $\HH_{s}$ in $\K^2$ given by
$$\HH_{s} := \{(y_0,y_1) \in \K^2; \ y_0 + s y_1 =0\}\, .$$
Notice that
$$(\F)^{-1}(\HH_{s}) = V(g_{s}) \, .$$

Following \cite{MJ} (Definition 2.1), we say that the family $(g_{s})$ is {\it $\Delta$-regular} (with respect to the stratification $\mathcal S$) if the line $\HH_0$ is not a limit of secant lines of the discriminant set $\Delta(\F)$ of $\F$ at $0 \in \K^2$. This means that there exist a neighborhood $U$ of $0$ in $\K^2$ and a real number $\delta>0$ such that 
$$\HH_s \cap \Delta(\F) \cap U \subset \{0\}$$
whenever $\|s\|<\delta$. 


Now we give:

\begin{defi}
Let $(g_s)$ be an isolated singularity family as above, and let $\e>0$ be a small real number. We say that a real number $\delta>0$ is a good parameter-radius for $(g_s)$ in $\B_\e$ if for every $s \in \D_\delta$ one has that:
\begin{itemize}
\item[(i)] $\Sigma(g_s) \cap \B_\e = \{0\}$;
\item[(ii)] $V(g_s) \setminus V(\varphi)$ intersects the sphere $\BS_\e$ transversally in $\K^n$, in the stratified sense;
\item[(iii)] $\HH_s \cap \Delta(\F) \cap U \subset \{0\}$.
\end{itemize}
\end{defi}

It follows from \cite{MJ} (Lemma 3.2 together with Proposition 4.4) that if $(g_s)$ has no coalescing of singular points, then for every $\e>0$ small enough there exists a good parameter-radius $\delta>0$ for $(g_s)$ in $\B_\e$.

Then the next proposition easily follows from the proof of (\cite{MJ}, Theorem 1.1).

\begin{prop} \label{prop_gpr}
Let $(g_s)$ be an isolated singularity family as above. If $\e>0$ is small enough and if $\delta>0$ is a good parameter-radius for $(g_s)$ in $\B_\e$, then for each $s \in \D_\delta$ there exists a homeomorphism 
$$h_s: V(g_s) \cap \B_\e \to V(g_0) \cap \B_\e$$ 
that fixes the origin.
\end{prop}

\section{Multi-parameters families of isolated singularities}  

Let $(X,0) \subset (\K^N,0)$ and $\mathcal S$ be as above. Let $f_{s_1, \dots, s_k}: (X,0) \to (\K,0)$ be an isolated singularity family of $\K$-analytic function-germs depending linearly on the parameters $s_1, \dots, s_k$, with $k>1$. This means that each $f_{s_1, \dots, s_k}$ has the form
$$f_{s_1, \dots, s_k}(z) = f_0(z) + s_1 \varphi_1(z) + \dots + s_k \varphi_k(z)  \, ,$$
where $f_0: (X,0) \to (\K,0)$ and $\varphi_i: (X,0) \to (\K,0)$ for $i=1, \dots, k$ are $\K$-analytic isolated singularity function-germs.

If $\K=\R$ set $\xi:=1$, and if $\K=\C$ set $\xi:=2$. Then let $\BS^{\xi k-1}$ denote the unitary sphere around $0$ in $\K^k$.

Now, for each point $r = (r_1, \dots, r_{k}) \in \BS^{\xi k-1}$ fixed, consider the one-parameter isolated singularity family $(g^r_{s})$ given by
$$g^r_{s}(z) := f_{s r_1, \dots, s r_k}(z) \, .$$
As before, we define the map-germ
$$
\begin{array}{cccc}
\F^r: & (X,0) & \to & (\K^2,0) \\
     & z & \mapsto & \big( f_0(z), r_1 \varphi_1(z) + \dots + r_k \varphi_k(z) \big) \\
\end{array}
\, ,$$
so that
$$(\F^r)^{-1}(\HH_{s}) = V(g^r_{s}) \, .$$

We say that the family $(f_{s_1, \dots, s_k})$ is $\Delta$-regular if for each $r \in \BS^{\xi k-1}$ the corresponding one-parameter family $(g^r_{s})$ is $\Delta$-regular. This means that for each $r \in \BS^{\xi k-1}$ there exist a neighborhood $U_r$ of $0$ in $\K^2$ and a real number $\delta_{r}>0$ such that 
$$\HH_s \cap \Delta(\F^r) \cap U_r \subset \{0\}$$
whenever $\|s\|<\delta_{r}$. 

We have:

\begin{lemma} \label{lemma_delta}
If the family $(f_{s_1, \dots, s_k})$ is $\Delta$-regular, then there exist a neighborhood $U$ of $0$ in $\K^2$ and a real number $\delta>0$ such that for every $r \in \BS^{\xi k-1}$ and for every $s \in \K$ with $\|s\|<\delta$ one has
$$\HH_s \cap \Delta(\F^r) \cap U \subset \{0\} \, .$$
\end{lemma}

\begin{proof}
By continuity, for each $r \in \BS^{\xi k-1}$ fixed, there exists some small open neighborhood $N_r$ of $r$ in $\BS^{\xi k-1}$ such that, for every $w \in N_r$, one has that $\HH_s \cap \Delta(\F^{w}) \cap U_r \subset \{0\}$ whenever $\|s\|<\delta_{r}$. Then the result follows from the fact that $\BS^{\xi k-1}$ is compact.
\end{proof}

Now suppose that the family $(f_{s_1, \dots, s_k})$ has no coalescing of singular points. Then it is clear that  there exists $\lambda>0$ sufficiently small such that for every $r \in \BS^{\xi k-1}$ and for every $\|s\| \leq \lambda$ one has that $\Sigma(g_{s}^r) \cap \B_{\e} = \{0\}$.

Moreover, it is not difficult to see that $\lambda>0$ can be chosen small enough so that it is also true that $V(g^r_s) \setminus V(r_1 \varphi_1 + \dots + r_k \varphi_k)$ intersects the sphere $\BS_\e$ transversally in $\K^n$, in the stratified sense, for every $s \in \D_\lambda$ and for every $r = (r_1, \dots, r_{k}) \in \BS^{\xi k-1}$.

So if we set $\delta' := \min\{\delta, \lambda\}$, it follows that for every $\e>0$ small enough, $\delta'$ is a good parameter-radius for $(g_s^r)$ in $\B_\e$, for every $r \in \BS^{\xi k-1}$. Then it follows from Proposition \ref{prop_gpr} that for each $s = (s_1, \dots, s_k) \in \K^k$ with $\|s\| < \delta'$ there exists a homeomorphism 
$$h_s: V(f_{s_1, \dots, s_k}) \cap \B_\e \to V(f_0) \cap \B_\e$$ 
that fixes the origin.

This proves Theorem \ref{theo_ctt}, and Theorem \ref{cor_ctt} follows easily. 

\begin{example} 
Recall the one-parameter isolated singularity family 
$$\mathtt f_s(x,y,z)=x^3+y^4+z^2+sx^4 + s^2y^5$$ 
on $X=\{x^2-y^2=0\} \subset \K^3$, as in Example \ref{example_1}. Consider the two-parameters family 
$$f_{s_1,s_2}(x,y,z)=x^3+y^4+z^2+s_1x^4 + s_2y^5$$ 
on $X$. Using Theorem \ref{theo_is}, one can easily see that the family $(f_{s_1,s_2})$ has no coalescing of singular points. So it follows from Theorem \ref{cor_ctt} that the family $(\mathtt f_s)$ has constant topological type.
\end{example}

Finally, Corollary \ref{cor_mcc} follows from Theorem \ref{cor_ctt} together with the following:

\begin{lemma}
In the case when $(X,0) = (\C^n,0)$, an isolated singularity family $(f_{s_1, \dots, s_k})$ has no coalescing of singular (critical) points if and only if it is $\mu$-constant.
\end{lemma}

\begin{proof}
First notice that in the case when $k=1$ the result is well-known (see \cite{Gr} for instance). So, in our case, where $k>1$, it follows that $(f_{s_1, \dots, s_k})$ has no coalescing of critical points if and only for each $r \in \BS^{2k-1}$ there exists a real number $\lambda_r>0$ such that $\mu(g_s^r)$ is constant for every $s \in \C$ with $\|s\|<\lambda$. But since $\BS^{2k-1}$ is compact, that happens if and only if there is a real number $\lambda>0$ such that $\mu(f_{s_1, \dots, s_k})$ is constant for every $s \in \C^k$ with $\|s\|<\lambda$.
\end{proof}

\begin{example}
Consider the two-parameters isolated singularity family $(f_{s_1,s_2})$ on $\C^2$ given by
$$f_{s_1,s_2}(x,y) := x^3 + y^2 + x^2y + s_1(x^2-x^2y) + s_2 x^2 \, .$$
Note that it is not $\mu$-constant. In fact, $\mu(f_0)=2$, while $\mu(f_{0,s_2})=1$ for $s_2 \neq 0$. However, the one-parameter family $(g_s)$ given by
$$g_s(x,y) := f_{s_1,-s_1}(x,y) = x^3 + y^2 + x^2y + s_1(-x^2y) $$
is clearly $\mu$-constant. 
\end{example}

Clearly, if the multi-parameters family $(f_{s_1, \dots, s_k})$ associated to a one-parameter family $(\mathtt f_s)$ on $\C^n$ is $\mu$-constant, then $(\mathtt f_s)$ is also $\mu$-constant. Thus our criterium for constancy of the topological type of $(\mathtt f_s)$ is weaker then the criterium suggested by Le-Ramanujam's Conjecture when $n=3$. So the natural question is:

\begin{question} 
Is there an example of a $\mu$-constant isolated singularity family $(\mathtt f_s)$ on $\C^3$ such that the corresponding multi-parameters family $(f_{s_1, \dots, s_k})$ is not $\mu$-constant? 
\end{question}

A negative answer to that question would imply that the $\mu$-constant conjecture is true.

\section{Deformations of isolated singularity function-germs on IDS}  

Now let $(X,0) \subset (\C^N,0)$ be the germ of a complex analytic isolated determinantal singularity (IDS), and let $\tilde f_s: (\C^N,0) \to (\C,0)$ be a family of complex function-germs depending holomorphically on the parameter $s = (s_1, \dots, s_k) \in \C^k$. This means that the map $\tilde F: (\C^N \times \C^k, (0,0)) \to (\C,0)$ defined by $F(x,s):= \tilde f_s(x)$ is holomorphic. The restriction of $\tilde F$ to $(X,0)$ defines a family of complex function-germs 
$$f_s: (X,0) \to (\C,0) \, .$$

If $(f_s)$ is an isolated singularity family, then for each $s \in (\C^k,0)$ there is a well-defined Milnor number $\mu(f_s)$ in the sense of \cite{BOT}. It is the number of points in the critical locus of a morsification $f_s': X' \to \C$ of $f_s$, where $X'$ is a smoothing of $X$ near the origin. 

In the particular case when $X$ is an ICIS, it is well-know that $\mu(f_s) \leq \mu(f_0)$ whenever $\|s\|$ is small enough. Moreover, the family $(f_s)$ has constant Milnor number if and only if it has no coalescing of singular points (\cite{COT} and \cite{Gr}). This means that there exist positive real numbers $\e>0$ and $\delta>0$ such that $0$ is the only critical point of $f_s$ in $\B_\e$, for every $s \in \BD_\delta$. These two facts come from the fact that when $X$ is an ICIS, the Milnor number $\mu(f_0,0)$ of $f_0$ at $0$ has an algebraic expression as the dimension of a quotient ring that is Cohen-Macaulay. Then the principle of conservation of number gives that 
\begin{equation}
\mu(f_0,0) = \sum_{x \in \Sigma(f_s)} \mu(f_s,x) \, ,
\end{equation}
where $\mu(f_s,x)$ denotes the Milnor number of $f_s$ at the singular point $x \in \Sigma(f_s)$.

When $X$ is an arbitrary IDS there is no such algebraic expression for the Milnor number of $f_s: (X,0) \to (\C,0)$, so one cannot apply the same argument. However, a simple topological argument gives the same equality. Precisely, we have:

\begin{lemma} \label{lemma}
Let $f_s: (X,0) \to (\C,0)$ be an isolated singularity family defined on an IDS. Then there exists $\delta>0$ such that for every $s \in \D_\delta$ equality $(1)$ above holds.
\end{lemma}

\begin{proof}
Let $f_s': X' \to \C$ be a morsification (in family) of $f_s$. Since $X$ is an IDS and since $f_0$ has an isolated singularity, we can choose $\e>0$ be small enough such that $X \cap \B_\e \setminus \{0\}$ is smooth and the restriction of $f_0$ to $X \cap \B_\e \setminus \{0\}$ is a submersion. We can also assume that $f_0'$ has exactly $\mu(f_0)$-many (Morse) critical points in $X' \cap \B_\e$. Then there exists $\delta>0$ sufficiently small such that for every $s \in \D_\delta$ one has that $f_s'$ has exactly $\mu(f_0)$-many critical points in $X' \cap \B_\e$. This follows from the fact that Morse points are stable under small perturbations. Thus equality $(1)$ above follows.
\end{proof}

As an immediate consequence, we have:

\begin{prop} \label{prop_pcn}
Let $f_s: (X,0) \to (\C,0)$ be an isolated singularity family defined on an IDS. Then there exists $\delta>0$ such that:
\begin{itemize}
\item[$(i)$] for every $s \in \D_\delta$ one has $\mu(f_s) \leq \mu(f_0)$.
\item[$(ii)$] $\mu(f_s) = \mu(f_0)$ for every $s \in \D_\delta$ if and only if $(f_s)$ has no coalescing of singular points. 
\end{itemize}
\end{prop}

On the one hand, Proposition \ref{prop_pcn} gives a characterization for the $\mu$-constancy of deformations of functions on an IDS, in the spirit of \cite{Gr} and \cite{COT}. 

On the other hand, if $(X,0) \subset (\C^N,0)$ is an IDS and if $f_s: (X,0) \to (\C,0)$ is an isolated singularity family depending linearly on the parameters $s_1, \dots, s_k \in \C$, Proposition \ref{prop_pcn} together with Theorem \ref{cor_ctt} give Theorem \ref{theo_3}.

Notice that Proposition \ref{prop_pcn} and Proposition \ref{prop} combined give the following Corollary, which is a good criterium to decide if the family $(f_s)$ is $\mu$-constant:

\begin{cor} \label{cor_2}
Let $\tilde f_s: (\C^N,0) \to (\C,0)$ be an isolated singularity family depending holomorphically on $s \in \C^k$. Let $(X,0) \subset (\C^N,0)$ be an IDS and suppose that the family $f_s: (X,0) \to (\C,0)$ given by restriction to $(X,0)$ is also an isolated singularity family. If $(\tilde f_s)$ is $\mu$-constant then $(f_s)$ is $\mu$-constant.
\end{cor}

\begin{example}
Let $X \subset \C^4$ be the $2$-dimensional IDS defined by the matrix
$$
\begin{pmatrix}
w & y & x \\
z & w & y \\
\end{pmatrix}
.$$
Consider the $\mu$-constant family $\tilde f_s: (\C^4,0) \to (\C,0)$ given by
$$\tilde f_s(x,y,z,w) := x^2+y^2+z^2+w^2+sx^3 \, ,$$
and let $f_s: (X,0) \to (\C,0)$ be its restriction to $X$. For each $s \in \C$ fixed, the singular points of $f_s$ are those points in $X$ at which all the $3 \times 3$-minors of the following matrix vanish:
$$
\begin{pmatrix}
\a(x) & 2y & 2z & 2w \\
0 & z & y & -2w \\
z & -w & x & -y \\
w & -2y & 0 & x \\
\end{pmatrix}
,$$
where $\a(x) := 2x+3sx^2$. Then some calculation shows that the singular locus of $f_s$ is the finite set 
$$\Sigma(f_s) = \{\a(x) = y=z=w=0\} = \left\{ \left( 0,0,0,0 \right), \left( -\frac{2}{3s},0,0,0 \right) \right\} \, .$$
Hence the family $(f_s)$ has no coalescing of singular points. Thus if follows from Proposition \ref{prop_pcn} that $f_s$ is a $\mu$-constant deformation of the function-germ $f_0: (X,0) \to (\C,0)$.

On the other hand, consider the family $g_s: (X,0) \to (\C,0)$ given by restriction to $X$ of the family $\tilde g_s: (\C^4,0) \to (\C,0)$ defined by
$$\tilde g_s(x,y,z,w) := x^3+y^2+z^2+w^2+sx^2 \, ,$$
which has non-constant Milnor number. Setting $\b(x):= 3x^2+2sx$ one has that 
$$\Sigma(g_s) = \{\b(x) = y=z=w=0\} = \left\{ \left( 0,0,0,0 \right), \left( -\frac{2s}{3},0,0,0 \right) \right\} \, .$$
Then the family $(g_s)$ has coalescing of singular points and thus it is not a $\mu$-constant deformation.
\end{example}



\end{document}